\documentclass[12pt,letterpaper]{amsart}
\usepackage{geometry}
\usepackage{hyperref,amsmath}
\usepackage{mathpazo}
\usepackage{amssymb,url,setspace,xcolor}
\newtheorem{theorem}{Theorem}

\newtheorem{proposition}{Proposition}

\newtheorem{corollary}{Corollary}

\onehalfspace

\title[Universality of Automorphisms]{Universality of Automorphisms on the Ball of Bounded Holomorphic Functions on the Polydisk}

\author{Timothy G. Clos}
\address[Timothy G. Clos]{Bowling Green State University, 
	Department of Mathematics and Statistics,  Bowling Green, Ohio 43403 }
\email{clost@bgsu.edu}

\date{\today}

\begin{document}
 
\maketitle

\begin{abstract}
    Given a sequence of automorphisms of the polydisk, we show that the associated composition semigroup homomorphisms on the ball of bounded holomorphic functions on the polydisk admit a universal inner function if a certain condition on the automorphisms are satisfied.
\end{abstract}

\section{Introduction}
Let $X$ be a separable, metrizable topological space.  A sequence of continuous maps $\{T_n:X\rightarrow X| n\geq 1\}$ is said to be universal if there exists $x\in X$ so that $\{T_n x:n\geq 1\}$ is dense in $X$.  Such an $x$ is called a universal element of $\{T_n\}$.  Let $\Omega\subset \mathbb{C}^n$ be a domain.  We let $H(\Omega)$ represent the set of all holomorphic functions on $\Omega$.  When equipped with the compact open topology, $H(\Omega)$ is a Frechet space.  Recall that a sequence $\{f_j\}\subset H(\Omega)$ converges to $f$ in the compact open topology if and only if $f_j\rightarrow f$ uniformly on compact subsets of $\Omega$ as $j\rightarrow \infty$.  We define \[\mathbb{D}^n:=\{(z_1,...,z_n)\in \mathbb{C}^n:|z_j|<1\,,j=1,2,...,n\}\] to be the unit polydisk and let \[\mathbb{T}^n:=\{(z_1,...,z_n)\in \mathbb{C}^n:|z_j|=1\,,j=1,...,n\}\] be the distinguished boundary.  Then $H^{\infty}(\mathbb{D}^n)$ denotes the space of bounded holomorphic functions on $\mathbb{D}^n$ equipped with the compact open topology.  Then we define the ball of $H^{\infty}(\mathbb{D}^n)$ to be 
\[\overline{\textit{Ball}}(H^{\infty}(\mathbb{D}^n)):=\{h\in H^{\infty}(\mathbb{D}^n):\|h\|_{L^{\infty}(\mathbb{D}^n)}\leq 1\}.\] 
We will show the non-Euclidean analog of Seidel-Walsh theorem (see \cite{GethShap} and \cite{SeidWal})  for the polydisk in $\mathbb{C}^n$ and then give necessary and sufficient conditions for a sequence of automorphisms of the disk and the polydisk to have a universal element on the ball of $H^{\infty}(\mathbb{D}^n)$. 
Since $\overline{\textit{Ball}}(H^{\infty}(\mathbb{D}^n))$ is a semigroup and not a vector space, one cannot use vector space techniques.  Instead, we show that the semigroup universality criterion in \cite{ChanWalm} is satisfied for certain sequences of composition operators (actually semigroup homomorphisms) on $\overline{\textit{Ball}}(H^{\infty}(\mathbb{D}^n))$.  On the unit ball in $\mathbb{C}^n$, \cite{Aron} studies whether a sequence of automorphisms on the unit ball in $\mathbb{C}^n$ admits a universal element.  On the polydisk, the main result in \cite{Chee} shows there exists a sequence of automorphisms of the polydisk with a universal element, but does not consider arbitrary sequences of automorphisms.  We give a natural condition on a sequence of automorphisms of the polydisk that ensures the associated composition semigroup homomorphisms admit a universal element.  The following theorem classifies automorphisms of the unit polydisk.  Let $S_n$ be the symmetry group of $n$ elements.  Recall the following theorem appearing in \cite{Rud}.

\begin{theorem}\cite{Rud}\label{thmrud}
Let $\phi$ be an automorphism of the polydisk $\mathbb{D}^n\subset \mathbb{C}^n$.  Then there exists $(\alpha_1,...,\alpha_n)\in \mathbb{D}^n$, $(\theta_1,...,\theta_n)\in \mathbb{R}^n$, and $p\in S_n$ so that
\[\phi(z_1,z_2,...,z_n)=\left(e^{i\theta_1}\frac{\alpha_1-z_{p(1)}}{1-\overline{\alpha_1}z_{p(1)}},...,e^{i\theta_n}\frac{\alpha_n-z_{p(n)}}{1-\overline{\alpha_n}z_{p(n)}}\right).\]
\end{theorem}

\section{Inner functions on the Polydisk}

Recall that on the disk in $\mathbb{C}$, an inner function is any bounded holomorphic function on the disk $g:\mathbb{D}\rightarrow \mathbb{C}$ so that \[\lim_{r\rightarrow 1^-}|f(re^{i\theta})|=1\] almost everywhere.  We can extend this definition to the polydisk in a natural way.  That is, we say a bounded holomorphic function on the polydisk $f:\mathbb{D}^n\rightarrow \mathbb{C}$ is an inner function if 
\[\lim_{r\rightarrow 1^-}|f(re^{i\theta_1},re^{i\theta_2},...,re^{i\theta_n})|=1\] almost everywhere.  We say an inner function $g:\mathbb{D}^n\rightarrow \mathbb{C}$ is a good inner function (see \cite{Chee}) if 
\[\int_{b\mathbb{D}^n}\log|g(r\zeta)|d\sigma(\zeta)\rightarrow 0\] as $r\rightarrow 1^-$.

The following proposition shows that if $f$ is a good inner function then for any automorphism of the polydisk, $\phi$, $f\circ \phi$ is also a good inner function.  That is, the action of the composition operator $C_{\phi}$ preserves good inner functions.

\begin{proposition}\label{prop1}
Let $g:\mathbb{D}^n\rightarrow \mathbb{C}$ be a good inner function and $\phi:\mathbb{D}^n\rightarrow \mathbb{D}^n$ be an automorphism of the polydisk.  Then, $g\circ \phi$ is a good inner function.  
\end{proposition}

\begin{proof}
First we will show if $g\in H^{\infty}(\mathbb{D}^n)$ is a good inner function and $\phi$ an automorphism of $\mathbb{D}^n$, then $g\circ \phi$ is an inner function.  Then $|g^*(w)|=1$ for almost every $w\in \mathbb{T}^n$ where $g^*$ denotes the boundary values (radial limits) of $g$ on $\mathbb{T}^n$.  By Theorem \ref{thmrud}, we may assume $\phi(z_1,...,z_n)=\left(\sigma_1(z_1),\sigma_2(z_2),...,\sigma_n(z_n)\right)$ where \[\sigma_j(z_j)=\frac{\alpha_j-\lambda_jz_j}{1-\overline{\alpha_j}\lambda_jz_j}\] is an inner function on $\mathbb{D}$ for $\alpha_j\in \mathbb{D}$ and $|\lambda_j|=1$ and $j=1,2,...,n$.  Then an application of \cite[Theorem 1.2.4]{Saw} allows us to write $|(g\circ \phi)^*|=|g^*\circ\phi^*|=1$ almost everywhere on $\mathbb{T}^n$.  Thus $g\circ \phi$ is inner.   

Now let $r<1$ be sufficiently large so that
\[r>\max\{|\alpha_j|:j=1,2,...,n\}.\]
Then the image $\phi(r\mathbb{T}^n)$ can be represented as \[\prod_{j=1}^n \mathbb{T}_{a_j(r),r_j(r)}.\]  Here, \[\mathbb{T}_{a_j(r),r_j(r)}\] is a circle centered at $a_j(r)$ with radius $r_j(r)$ where $a_j(r)\rightarrow 0$ and $r_j(r)\rightarrow 1$ as $r\rightarrow 1^-$.  For some fixed $M>\sup\{|J(\phi)(z)|^{-1}:z\in \overline{\mathbb{D}^n}\}$ we have,   
\begin{align}
&|\int_{\mathbb{T}^n}\log|g\circ\phi(r\zeta)|d\sigma(\zeta)|\\
&=|\int_{(r\mathbb{T}^n)}\log|g\circ\phi(\zeta)|r^{-n}d\sigma(\zeta)|\\
&=|\int_{\phi((r\mathbb{T}^n))}\log|g(w)||J(\phi)(w)|^{-1}r^{-n}d\sigma(w)|\\
&=|\int_{\prod_{j=1}^n \mathbb{T}_{a_j(r),r_j(r)}}\log|g(w)||J(\phi)(w)|^{-1}r^{-n}d\sigma(w)|\\
&\leq\int_{0}^{2\pi}\int_0^{2\pi}...\int_0^{2\pi}|\log|g(a_1(r)+r_1(r)e^{i\theta_1},...,a_n(r)+r_n(r)e^{i\theta_n})|\\
&-\log|g(r_1(r)e^{i\theta_1},...,r_n(r)e^{i\theta_n})||Mr^{-n}d\theta_1...d\theta_n\\
&+ \int_0^{2\pi}...\int_{0}^{2\pi}|\log|g(r_1(r)e^{i\theta_1},...,r_n(r)e^{i\theta_n})|-\log|g(Re^{i\theta_1},...,Re^{i\theta_n})||Mr^{-n} d\theta_1...d\theta_n\\
&+\int_{\zeta\in \mathbb{T}^n}-\log|g(R\zeta)|r^{-n}M d\sigma(\zeta)
\end{align}

Then using a continuity argument and using the Lebesgue dominated convergence theorem, one can make lines 5 and 6 sufficiently small when $r<1$ is sufficiently large and $R>r$ is sufficiently close to $r$.  Furthermore, using the assumption that $g$ is a good inner function, one can make the line 7 integral arbitrarily small if $r<1$ is sufficiently large.  Then the line 8 integral is arbitrarily small for all $r<1$ sufficiently large since $g$ is a good inner function.

Thus we have shown that \[|\int_{\mathbb{T}^n}\log|g\circ\phi(r\zeta)|d\sigma(\zeta)|\rightarrow 0\] as $r\rightarrow 1^-$ for any automorphism of the polydisk $\phi$ and good inner function $g$.

\end{proof}

The following proposition uses a result in \cite{Rud} about the density of inner functions continuous up to $\overline{\mathbb{D}^n}$ in the ball of $H^{\infty}(\mathbb{D}^n)$.

\begin{proposition}\label{denseprop}
Let $\overline{\textit{Ball}}(H^{\infty}(\mathbb{D}^n))$ be the ball of $H^{\infty}(\mathbb{D}^n)$ as defined previously.  Suppose $\lambda:=(\lambda_1,\lambda_2,...,\lambda_n)\in \mathbb{T}^n$.  Then there exists a collection of inner functions $\{G_{j,\lambda}\}_{j\in \mathbb{N}}\subset \overline{Ball}(H^{\infty}(\mathbb{D}^n))$ so that  
\begin{enumerate}
    \item $G_{j,\lambda}\in C(\overline{\mathbb{D}^n})$ for all $j\in \mathbb{N}$.
    \item $\{G_{j,\lambda}\}_{j\in \mathbb{N}}$ is dense in $\overline{Ball}(H^{\infty}(\mathbb{D}^n))$ equipped with the compact-open topology.
    \item $G_{j,\lambda}(\lambda_1,\lambda_2,...,\lambda_n)=1$ for all $j\in \mathbb{N}$.
\end{enumerate}
\end{proposition} 

\begin{proof}
By \cite[Theorem 5.5.1]{Rud}, there exists a collection of inner functions $\{A_j\}_{j\in \mathbb{N}}\subset \overline{\textit{Ball}}(H^{\infty}(\mathbb{D}^n))$ that is dense in $\overline{Ball}(H^{\infty}(\mathbb{D}^n))$ equipped with the compact open topology and $A_j\in C(\overline{\mathbb{D}^n})$ for all $j\in \mathbb{N}$.  To construct the sequence $G_{j,\lambda}$ we modify the sequence $A_j$ as follows.  By \cite{Chee}, (see also \cite{Heins}) there exists $\Psi_j\in C^{\infty}(\overline{\mathbb{D}^n})\cap H^{\infty}(\mathbb{D}^n)$ so that for all $j\in \mathbb{N}$, $\Psi_j$ is inner, \[\Psi_j(1,1,...,1)=\frac{1}{A_j(1,1,...,1)},\] and $\Psi_j(0,0,...,0)=1-2^{-j}$.  We set 
\[\Psi_{j,\lambda}(z_1,...,z_n):=\Psi_j(\lambda_1^{-1}z_1,...,\lambda_n^{-1}z_n),\]
\[A_{j,\lambda}(z_1,...,z_n):=A_j(\lambda_1^{-1}z_1,...,\lambda_n^{-1}z_n),\]
and $G_{j,\lambda}:=A_{j,\lambda}\Psi_{j,\lambda}$.  Then for all $j\in \mathbb{N}$, $G_{j,\lambda}(\lambda_1,...,\lambda_n)=1$ and $G_{j,\lambda}\in C(\overline{\mathbb{D}^n})$.  Let $f\in \overline{\textit{Ball}}(H^{\infty}(\mathbb{D}^n)) $.  Then there exists a subsequence $A_{j_k}$ so that $A_{j_k}\rightarrow f(\lambda_1z_1,...,\lambda_nz_n)$ uniformly on compact subsets of $\mathbb{D}^n$ as $k\rightarrow \infty$.  So $A_{j_k,\lambda}\rightarrow f$ uniformly on compact subsets of $\mathbb{D}^n$ as $k\rightarrow \infty$.  Then $\{\Psi_{j_k,\lambda}\}$ is a normal family, so by passing to a subsequence if necessary, we may assume by the maximum modulus principle that $\Psi_{j_k,\lambda}\rightarrow 1$ uniformly on compact subsets of $\mathbb{D}^n$ as $k\rightarrow \infty$.  So $G_{j_k,\lambda}\rightarrow f$ uniformly on compact subsets of $\mathbb{D}^n$ as $k\rightarrow \infty$.  Thus 
$\{G_{j,\lambda}:j\in \mathbb{N}\}$ is dense in $\overline{\textit{Ball}}(H^{\infty}(\mathbb{D}^n))$ equipped with the compact open topology.

\end{proof}

\begin{proposition}\label{prop3}
Let $\phi$ be an automorphism of $\mathbb{D}^n$, and $C_{\phi}:H(\mathbb{D}^n)\rightarrow H(\mathbb{D}^n)$ be its associated composition operator.  Then $C_{\phi}$ restricted to $\overline{\textit{Ball}}(H^{\infty}(\mathbb{D}^n))$ is surjective onto $\overline{\textit{Ball}}(H^{\infty}(\mathbb{D}^n))$. 
\end{proposition}

\begin{proof}
Let $f\in \overline{\textit{Ball}}(H^{\infty}(\mathbb{D}^n))$.  Then it is clear that $C_{\phi^{-1}}f\in \overline{\textit{Ball}}(H^{\infty}(\mathbb{D}^n)) $ and $f=C_{\phi}\left(C_{\phi^{-1}}f\right)$.
\end{proof}

\section{Universality Criterion for Semigroups}

We use the semigroup analog of the following theorem in a significant way, as vector space techniques are not applicable to the ball of $H^{\infty}(\mathbb{D}^n)$.  

\begin{theorem}\cite{GethShap}\label{geth}
Suppose $T$ is a continuous linear operator on a separable $F$-space $X$.  Suppose there exists a dense subset $\mathcal{D}$ of $X$ and a right inverse $S$ for $T$ so that $\|T^n x\|\rightarrow 0$ and $\|S^n x\|\rightarrow 0$ for all $x\in \mathcal{D}$, then $X$ has $T$-universal vectors.

\end{theorem}

The following is a consequence of Theorem \ref{geth} as mentioned in \cite[pp. 283]{GethShap}.

\begin{corollary}\cite{GethShap}
Suppose $\mathcal{D}$ is a dense subset of $X$ and $\{T_j\}$ is a sequence of continuous linear operators on $X$ for which $T_j\rightarrow 0$ pointwise on $\mathcal{D}$ suppose for each $j$ the operator $T_j$ has a right inverse $S_j$ and $S_j\rightarrow 0$ pointwise on $\mathcal{D}$.  Then the set $\{T_j x:j\geq 0\}$ is dense in $X$ for a dense $G_{\delta}$ set of vectors $x\in X$.
\end{corollary}

This semigroup analog appears in \cite{ChanWalm} and is stated here for the convenience of the reader.
\begin{theorem}\cite{ChanWalm}\label{One}
Let $X$ be a separable, metrizable, complete, topological semigroup with identity element $e$.  Suppose $\{T_n:X\rightarrow X\}_{n\in \mathbb{N}}$ is a collection of continuous semigroup homomorphisms.  If there exist dense sets $\mathcal{D}_0\subset X$, $\mathcal{D}_1\subset X$, and a sequence of mappings $\{R_n:\mathcal{D}_1\rightarrow X\}_{n\in \mathbb{N}}$ so that 
\begin{enumerate}
    \item $T_n\mathcal{D}_0\rightarrow e$ as $n\rightarrow \infty$.
    \item $R_n\mathcal{D}_1\rightarrow e$ as $n\rightarrow \infty$.
    \item $T_nR_nf\rightarrow f$ for all $f\in \mathcal{D}_1$.
\end{enumerate}
Then there exists $\{x_j\}\subset \mathcal{D}_0$ so that $x:=\prod_{j}x_j$ is universal for $\{T_n\}$.  That is, the orbit $\{T_n x:n\in \mathbb{N}\}$ is dense in $X$.
\end{theorem}

\section{Main Result and Proof}

Now we will study when an arbitrary sequence of automorphisms of the polydisk has a universal element.

This next theorem is a new result and improves \cite{Chee}, and is valid for the ball of $H^{\infty}(\mathbb{D}^n)$ for $n=1,2,...$.

\begin{theorem}
Let \[\phi_k(z_1,...,z_n):=\left(e^{i\theta^k_1}\frac{\alpha^k_1-z_{p_k(1)}}{1-\overline{\alpha^k_1}z_{p_k(1)}},...,e^{i\theta^k_n}\frac{\alpha^k_n-z_{p_k(n)}}{1-\overline{\alpha^k_n}z_{p_k(n)}}\right)\] be a sequence of automorphisms of the polydisk $\mathbb{D}^n\subset \mathbb{C}^n$.  If $(\alpha_1^k,...,\alpha_n^k)\rightarrow (\alpha_1,...,\alpha_n)\in \mathbb{T}^n$ as $k\rightarrow \infty$ then there exists an inner function $x$ so that $\{x\circ \phi_k:k\in \mathbb{N}\}$ is dense in $\overline{\textit{Ball}}(H^{\infty}(\mathbb{D}^n))$.

Furthermore, the universal element $x$ has the form 
\[x=\prod_{j=1}^{\infty}x_j\] for some sequence of continuous (up to $\overline{\mathbb{D}^n}$) inner functions $\{x_j\}_{j\in \mathbb{N}}$.

\end{theorem}
\begin{proof}
We first note that the theorem is actually stronger than is stated in that we only require a subsequence of $(\alpha_1^k,...,\alpha_n^k)$ to converge to $\mathbb{T}^n$ as $k\rightarrow \infty$.  So, without loss of generality, we may assume $(\alpha_1^k,...,\alpha_n^k)\rightarrow (\alpha_1,...,\alpha_n) \in \mathbb{T}^n$ as $k\rightarrow \infty$.
Then by Proposition \ref{prop3}, $C_{\phi_k}:\overline{\textit{Ball}}(H^{\infty}(\mathbb{D}^n))\rightarrow \overline{\textit{Ball}}(H^{\infty}(\mathbb{D}^n))$ is surjective, so has a right inverse, denoted by $C_{\phi_k}^{-1}$. 
We let $\phi_{k_l}$ be a subsequence of $\phi_k$ with the following properties.
\begin{enumerate}
    \item $(\theta_1^{k_l},...,\theta_n^{k_l})\rightarrow (\theta_1,...,\theta_n)$ as $l\rightarrow \infty$.
    \item $p_{k_l}=p_{k_{l+1}}:=p$ for all $l\in \mathbb{N}$.  
    
\end{enumerate}

We note that \[\phi_{k_l}^{-1}=\left(e^{i\theta^{k_l}_{p^{-1}(1)}}\frac{a^{k_l}_{p^{-1}(1)}-z_{p^{-1}(1)}}{1-\overline{a^{k_l}_{p^{-1}(1)}}z_{p^{-1}(1)}},...,e^{i\theta^{k_l}_{p^{-1}(n)}}\frac{a^{k_l}_{p^{-1}(n)}-z_{p^{-1}(n)}}{1-\overline{a^{k_l}_{p^{-1}(n)}}z_{p^{-1}(n)}}\right).\]
Since $\phi_{k_l}$ and $\phi_{k_l}^{-1}$ are not necessarily equal, we need to define dense (in the ball of $H^{\infty}(\mathbb{D}^n)$) sets $\mathcal{D}_0$ and $\mathcal{D}_1$ so that $C_{\phi_{k_l}}\mathcal{D}_0\rightarrow 1$ and $C_{\phi_{k_l}^{-1}}\mathcal{D}_1\rightarrow 1$ as $l\rightarrow \infty$. Then we can apply the semigroup universality criterion seen in Theorem \ref{One}.

One can show that for any $g\in C(\overline{\mathbb{D}^n})$, $g\circ \phi_{k_l}\rightarrow g(\lambda)$ uniformly on compact subsets of $\mathbb{D}^n$ as $l\rightarrow \infty$ where  \[\lambda:=(e^{i\theta_1}\alpha_1,...,e^{i\theta_n}\alpha_n)=\lim_{l\rightarrow \infty}\left(e^{i\theta_1^{k_l}}\alpha_1^{k_l},...,e^{i\theta_n^{k_l}}\alpha_n^{k_l}\right).\]  Furthermore, $g\circ \phi_{k_l}^{-1} \rightarrow g(\gamma)$ uniformly on compact subsets of $\mathbb{D}^n$ as $l\rightarrow \infty$, where \[\gamma:=(e^{i\theta_{p^{-1}(1)}}\alpha_{p^{-1}(1)},...,e^{i\theta_{p^{-1}(n)}}\alpha_{p^{-1}(n)})=\lim_{l\rightarrow \infty}\left(e^{i\theta_{p^{-1}(1)}^{k_l}}\alpha_{p^{-1}(1)}^{k_l},...,e^{i\theta_{p^{-1}(n)}^{k_l}}\alpha_{p^{-1}(n)}^{k_l}\right).\] 

For $G_{j,\lambda}$ and $G_{j,\gamma}$ as defined in Proposition \ref{denseprop}, we define \[\mathcal{D}_0:=\{G_{j,\lambda}:j\in \mathbb{N}\}\] and \[\mathcal{D}_1:=\{G_{j,\gamma}:j\in \mathbb{N}\}.\]

Then if $C_{\phi_{k_l}}$ is defined as the composition operator on $H(\mathbb{D}^n)$, one can show that 
\[C_{\phi_{k_l}}\mathcal{D}_0\rightarrow 1\] and \[C_{\phi_{k_l}}^{-1}\mathcal{D}_1\rightarrow 1\] uniformly on compact subsets of $\mathbb{D}^n$ as $l\rightarrow \infty$.  By Proposition \ref{denseprop}, $\mathcal{D}_0$ and $\mathcal{D}_1$ are dense in the ball of $H^{\infty}(\mathbb{D}^n)$ equipped with the compact-open topology.  Thus we can conclude by Theorem \ref{One} that there exists a universal inner function $x$ for $C_{\phi_k}$.  Furthermore,
$x=\prod_{j\in \mathbb{N}}x_j$ for some $\{x_j\}\subset \mathcal{D}_0$, implying $x_j\in C(\overline{\mathbb{D}^n})$ for all $j\in \mathbb{N}$.

\end{proof}

\section{Acknowledgements}
I wish to thank Kit Chan for useful discussions and comments on an earlier draft of this paper.  I also wish to thank the anonymous referee for the helpful comments.

\bibliographystyle{amsalpha}
\bibliography{timclos}

\end{document}